\documentclass[12pt,a4paper,oneside]{amsart}
\usepackage{color,amssymb,latexsym,amsfonts,textcomp,fancyhdr,calc,graphicx}
\usepackage{amsmath,amstext,amsthm,amssymb,amsxtra}
\usepackage[left=3cm,right=3cm,bottom=3cm,top=3cm]{geometry} 
\usepackage{hyperref}
\usepackage{txfonts,pxfonts,tikz} 
\usepackage[T1]{fontenc}
\usepackage{units}
\usepackage[backend=bibtex,style=alphabetic,citestyle=alphabetic]{biblatex}
\newtheorem{theorem}{Theorem}
\newtheorem{corollary}[theorem]{Corollary}
\newtheorem{lemma}{Lemma}

\newtheorem{definition}{Definition}[section]
\theoremstyle{definition}

\newcommand{\beql}[1]{\begin{equation}\label{#1}}
\newcommand{\eeq}{\end{equation}}
\newcommand{\comment}[1]{}

\newcommand{\Abs}[1]{{\left|{#1}\right|}}

\newcommand{\Set}[1]{{\left\{{#1}\right\}}}

\newcommand{\RR}{{\mathbb R}}

\newcommand{\ZZ}{{\mathbb Z}}

\newcommand{\inner}[2]{{\langle #1, #2 \rangle}}

\newcommand{\dens}{{\rm dens\,}}
\newcommand{\udens}{\overline{{\rm dens\,}}}

\newcommand{\ft}[1]{\widehat{#1}}

\newcounter{rem}
\newcounter{step}
\newcounter{mysec}
\newcounter{mysubsec}[mysec]
\newcounter{othm}
\def\theothm{\Alph{othm}}

\begin{document}

\title{Packing near the tiling density and exponential bases for product domains}
\author{Mihail N. Kolountzakis}
\thanks{This work has been partially supported by the ``Aristeia II'' action (Project
FOURIERDIG) of the operational program Education and Lifelong Learning
and is co-funded by the European Social Fund and Greek national resources.}
\address{Department of Mathematics and Applied Mathematics, University of Crete, Voutes Campus, GR-700 13, Heraklion, Crete, Greece}
\email{kolount@gmail.com}

\begin{abstract}
A set $\Omega$ in a locally compact abelian group is called spectral if $L^2(\Omega)$ has an orthogonal basis of
group characters.
An important problem, connected with the so-called Spectral Set Conjecture (saying that $\Omega$ is spectral if and only if
a collection of translates of $\Omega$ can partition the group), is the question
of whether the spectrality of a product set $\Omega = A \times B$, in a product group, implies the spectrality of
the factors $A$ and $B$.
Recently Greenfeld and Lev proved that if $I$ is an interval and $\Omega \subseteq \RR^d$ then the spectrality
of $I \times \Omega$ implies the spectrality of $\Omega$.
We give a different proof of this fact by first proving a result about packings of high density implying the existence
of tilings by translates of a function.
This allows us to improve the result to a wider collection of product sets than those dealt with by Greenfeld and Lev.
For instance when $A$ is a union of two intervals in $\RR$ then we show that the spectrality of $A \times \Omega$ implies
the spectrality of both $A$ and $\Omega$.
\end{abstract}

\maketitle


\noindent{\bf AMS Primary Classification:} 42C99, 52C22

\tableofcontents

\section{Introduction}
\label{sec:intro}

\subsection{A review of the Fuglede problem on spectral sets and tiles}
\label{sec:fuglede-review}

Let $\Omega \subseteq \RR^d$ be a bounded measurable set.
The concept of a spectrum of $\Omega$ that we deal with in this paper
was introduced by Fuglede \cite{fuglede1974operators}. 

\begin{definition}
A set $\Lambda \subseteq \RR^d$
is called a {\em spectrum} of $\Omega \subseteq \RR^d$ (and $\Omega$ is said to be a {\em spectral set})
if the set of exponentials
$$
E(\Lambda) = \Set{e_\lambda(x)=e^{2\pi i \lambda\cdot x}:\ \lambda\in\Lambda}
$$
is a complete orthogonal set in $L^2(\Omega)$
under the inner product $\inner{f}{g} = \int_\Omega f \overline{g}$.
\end{definition}

At least for bounded $\Omega$
it is easy to see (see, for instance, \cite{kolountzakis2004milano}) that the orthogonality of $E(\Lambda)$
is equivalent to the {\em packing condition}
\beql{packing-condition}
\sum_{\lambda\in\Lambda}\Abs{\ft{\chi_\Omega}}^2(x-\lambda) \le \Abs{\Omega}^2,\ \ \mbox{a.e. ($x$)},
\eeq
as well as to the condition
\beql{zeros-condition}
\Lambda-\Lambda \subseteq \Set{0} \cup \Set{\ft{\chi_\Omega}=0}.
\eeq
Here $\chi_\Omega$ is the indicator function of $\Omega$.

The orthogonality and completeness of $E(\Lambda)$ is in turn equivalent to the {\em tiling condition}
\beql{tiling-condition}
\sum_{\lambda\in\Lambda}\Abs{\ft{\chi_\Omega}}^2(x-\lambda) = \Abs{\Omega}^2,\ \ \mbox{a.e. ($x$)}.
\eeq
These equivalent conditions follow from the identity
$$
\inner{e_\lambda}{e_\mu} = \int_\Omega e_\lambda \overline{e_\mu} = \ft{\chi_\Omega}(\mu-\lambda)
$$
and from the density of trigonometric polynomials in $L^2(\Omega)$.
Condition \eqref{packing-condition} is roughly expressing the validity of Bessel's inequality for the
system of exponentials $E(\Lambda)$ while condition \eqref{tiling-condition} says that Bessel's inequality
holds as equality.

If $\Lambda$ is a spectrum of $\Omega$ then so is any translate of $\Lambda$ but there may be other spectra as well.

{\em Example:} If $Q_1 = (-1/2, 1/2)^d$ is the cube of
unit volume in $\RR^d$ then
$\ZZ^d$ is a spectrum of $Q_1$.
Let us remark here that
there are spectra of $Q_1$ which are very different from translates of the lattice $\ZZ^d$
\cite{iosevich1998spectral,lagarias2000orthonormal,kolountzakis2000packing}.

Research on spectral sets
\cite{lagarias1997spectral,laba2002spectral,laba2001twointervals}
has been influenced for many years by a conjecture of Fuglede
\cite{fuglede1974operators},
sometimes called the {\em Spectral Set Conjecture},
which stated that a set $\Omega$ is spectral if and only
if it tiles by translation. A set $\Omega$ tiles by translation (or just tiles, for this paper) if
we can translate copies of $\Omega$ around and fill space without overlaps.
More precisely there exists a set $S \subseteq \RR^d$ such that
\beql{tiling}
\sum_{s\in S} \chi_\Omega(x-s) = 1,\ \ \mbox{a.e. ($x$)}.
\eeq

One can generalize naturally the notion of translational tiling from sets to functions
by saying that a nonnegative $f \in L^1(\RR^d)$ tiles when translated at the locations $S$
if $\sum_{s\in S} f(x-s) = \ell$ for almost every $x\in\RR^d$
(the constant $\ell$ is called the {\em level} of the tiling).
Thus the question of spectrality for a set $\Omega$ is essentially a tiling question
for the function $\Abs{\ft{\chi_\Omega}}^2$.
Because of the equivalent condition \eqref{tiling-condition}
one can now restate the Fuglede Conjecture as the equivalence
(all tilings are by translation only in this paper)
\beql{fuglede-conjecture}
\Omega \mbox{ tiles $\RR^d$ at level 1} \Longleftrightarrow
\Abs{\ft{\chi_\Omega}}^2 \mbox{ tiles $\RR^d$ at level $\Abs{\Omega}^2$}.
\eeq
The equivalence \eqref{fuglede-conjecture} is known, from the time of Fuglede's
paper \cite{fuglede1974operators}, to be true if one adds the word {\em lattice} to both sides
(that is, lattice tiles are the same as sets with a lattice spectrum and the dual of any tiling lattice
is a spectrum).

The full conjecture \eqref{fuglede-conjecture} is, however, now known to be false in both directions if $d\ge 3$
\cite{tao2004fuglede,matolcsi2005fuglede4dim,kolountzakis2006hadamard,kolountzakis2006tiles,farkas2006onfuglede,farkas2006tiles},
but remains open in dimensions $1$ and $2$ and it is not out of the question
that the conjecture is true in all dimensions if one restricts the domain $\Omega$ to be convex.

It is known
that the direction ``tiling $\Rightarrow$ spectrality'' is true in the case of convex domains;
see for instance \cite{kolountzakis2004milano}.
In the direction ``spectrality $\Rightarrow$ tiling'' it was proved in \cite{iosevich2003fuglede}
that in $\RR^2$ every spectral convex domain must be a polygon and also tiles the plane (this
restricts the polygon to be either a parallelogram or a symmetric hexagon).
In a major recent result Greenfeld and Lev \cite{greenfeld2016fuglede} proved 
that any convex polytope in $\RR^3$ which is spectral must have symmetric facets (a property that also holds for convex polytopes that
tile) and, furthermore, it admits tilings by translation. This makes the validity of Fuglede's conjecture for convex domains in $\RR^3$
very close to being proved (it has long been known \cite{iosevich2001convexbodies} that convex bodies in $\RR^d$ with a point of
curvature are not spectral, and all that's missing is a proof that any spectral convex domain in $\RR^3$ is necessarily a polytope). 

\subsection{Tiling and spectrality for products and factors}
\label{sec:products}
To prove the result in \cite{greenfeld2016fuglede} Greenfeld and Lev first proved \cite{greenfeld2016spectrality} that
if $I \subseteq \RR$ is an interval and $A \subseteq \RR^d$ is such that $I \times A \subseteq \RR^{d+1}$ is spectral
then the set $A$ must itself be spectral.
Our main result in this paper is the extension of this result of Greenfeld and Lev to the case where $I$ is the union of two intervals
(see Corollary \ref{cor:two-intervals}, which comes from the more general Theorem \ref{th:cylinders}).

Our method is different from that followed in \cite{greenfeld2016spectrality}.
Instead of making a series of modifications to the spectrum of $I \times A$ (as in \cite{greenfeld2016spectrality}) that bring the
spectrum to a form that enables one
to read a spectrum of $A$ from the modified spectrum of $I \times A$, we are basing our approach on Theorem \ref{th:main} which
roughly says that if one can achieve packings of an object with density arbitrarily close to the tiling density then the object tiles.
This is a natural statement which is not hard to prove (but requires some care).

It is easy to see that whenever $A \times B$ tiles $\RR^m \times \RR^n$ by translation then $A$ tiles $\RR^m$ and $B$ tiles $\RR^n$.
Indeed, assume that $\sum_{s \in S} \chi_{A \times B}((x,y)-s) = 1$ for almost all $(x,y) \in \RR^{m+n}$.
By Fubini's theorem there is $x \in \RR^m$ such that the above function is 1 for almost all $y \in \RR^n$.
This means exactly that the function $\chi_B$ tiles $\RR^n$ when translated at the locations
$$
\pi_2 \Set{(s_1, s_2) \in S: x-s_1 \in A},
$$
where $\pi_2(x,y) = y$.
More intuitively, if a product set $A \times B$ tiles space $\RR^m\times\RR^n$ then almost every translate of $\Set{0}\times\RR^n$
is tiled by copies (translates) of $B$.

It is also very easy to see that if $A$ and $B$ are tiles then so is $A \times B$ and if $A$ and $B$ are spectral
then so is $A \times B$.

Thus the (still unknown) implication
\beql{product-to-factors}
A \times B \ \ \text{spectral} \Rightarrow A \ \ \text{spectral and } B \ \ \text{spectral}
\eeq
is very important for the Fuglede conjecture.
For if we suppose the ``spectral $\Rightarrow$ tiling'' half of the Fuglede conjecture to be true
in $\RR^{m+n}$ and the ``tiling $\Rightarrow$ spectral'' half to be true in $\RR^m$ and in $\RR^n$
then it follows that if $A \times B \subseteq \RR^m\times\RR^n$ is spectral then so
are $A\subseteq\RR^m$ and $B\subseteq\RR^n$.

Thus if one finds a counterexample to \eqref{product-to-factors} in dimensions $m=n=1$
this will imply the failure of the ``spectral $\Rightarrow$ tiling'' half in $\RR^2$ {\em or}
the ``tiling $\Rightarrow$ spectral'' half in $\RR$, without distinguishing which one fails.
But in any case this would imply that the Fuglede conjecture
(as the conjunction of the two implications ``spectral $\Rightarrow$ tiling'' and ``tiling $\Rightarrow$ spectral'')
fails in $\RR^2$.
The importance of finding out if \eqref{product-to-factors} holds is evident.
The results of this paper and the result in \cite{greenfeld2016spectrality} may be viewed as proof of \eqref{product-to-factors}
under extra assumptions on one of the factors.

\subsection{Notation and some definitions.}
\label{sec:notations}

We write
$$
Q_R = [-R/2, R/2]^d
$$
for the $0$-centered cube of side length $R$.

If $\Lambda \subseteq \RR^d$ is a discrete set then we write
$$
\delta_\Lambda = \sum_{\lambda \in \Lambda} \delta_\lambda
$$
for the locally finite measure that consists of a unit point mass at each point of $\Lambda$.
With this notation we can write
$$
\sum_{\lambda \in \Lambda} f(x-\lambda) = f*\delta_\Lambda(x).
$$

If $f \ge 0$ is a function (often an indicator function) and $\Lambda$ is a set in $\RR^d$ we say
that $f$ packs with $\Lambda$ at level $\ell>0$ if
\beql{packing}
\sum_{\lambda \in \Lambda} f(x-\lambda) \le \ell,\ \ \ \ (\text{for almost every $x \in \RR^d$}).
\eeq
If \eqref{packing} holds almost everywhere as equality we say that $f$ tiles with $\Lambda$ at level $\ell$.
Often we say that $f+\Lambda$ is a packing or a tiling to denote this situation.

Whenever we speak of packing or tiling without mentioning the level of the packing or the tiling we imply that the level is 1.

Our definition of density and upper density of a (usually discrete) set in $\RR^d$ is the usual asymptotic, translation-invariant one.
The {\em upper density} of a set $\Lambda$ is the quantity
$$
\limsup_{R \to \infty} \sup_{x \in \RR^d} \frac{\Abs{\Lambda \cap \big(x+[-R/2, R/2]^d\big)}}{R^d},
$$
(here $\Abs{\cdot}$ denotes cardinality)
with the corresponding $\liminf$ being the {\em lower density}. If the upper and lower density are equal then we call this
the {\em density} of the set.

A set $\Lambda \subseteq \RR^d$ is called {\em uniformly discrete} if there is $\delta>0$ such that
$\Abs{\lambda_1 - \lambda_2} > \delta$ whenever $\lambda_1, \lambda_2 \in \Lambda$ are different.

Following \cite{greenfeld2016spectrality} we define the {\em weak convergence} of the sets $\Lambda^n \subseteq \RR^d$
to the set $\Lambda \subseteq \RR^d$. If there is $\delta>0$ which is a seperating constant for all the $\Lambda^n$ and $\Lambda$
then we say that the $\Lambda^n$ converge weakly to $\Lambda$ if
for every $\epsilon, R > 0$ there is $N$ such that for all $n \ge N$ we have
$$
\Lambda^n \cap Q_R \subseteq \Lambda+Q_\epsilon\ \ \text{and}\ \ \Lambda \cap Q_R \subseteq \Lambda^n+Q_\epsilon.
$$

If $\lambda = (x,y) \in A\times B$ we write $x = \pi_1\lambda \in A$ and $b = \pi_2\lambda \in B$.

\section{The right packing density guarantees the existence of tilings}
\label{sec:tilings-from-densit}

Our main result for this section roughly says that if an object can pack space arbitrarily close to tiling level
then it can actually tile space exactly. This is essentially a compactness phenomenon.
\begin{theorem}\label{th:main}
If $f$ satisfies
\beql{assumptions}
f \in L^1(\RR^d),\ \  f \ge 0,\ \  \int f = 1,\ \  f>1/2 \ \ \text{on a set of positive measure},
\eeq
and has packings $f+\Lambda$ of upper density $\udens \Lambda$ arbirarily close to 1
then it admits tilings.
\end{theorem}

\noindent
{\bf Remark:}
Notice that if $f+\Lambda$ is a packing and $f$ satisfies \eqref{assumptions}
then there is a constant $\delta_0$, which depends only on $f$, such that
any two points of $\Lambda$ are at least $\delta_0$ apart.

We organize the proof in a few lemmas.

\begin{lemma}\label{lm:density-integral}
Assume \eqref{assumptions}.
Suppose $f$ has packings $f+\Lambda^n$ such that $\udens \Lambda^n \to 1$.
Then for any $R>0$ and for any $\epsilon>0$ there is a packing set of translates $\Lambda$ such that
$$
\int_{Q_R} f*\delta_{\Lambda} \ge \Abs{Q_R} - \epsilon.
$$
\end{lemma}

\begin{proof}
The statement is equivalent to that for any $R>0$ and $0 < \rho < 1$ there is a packing set $\Lambda$
such that
\beql{density-ratio}
\int_{Q_R} f*\delta_{\Lambda} \ge \rho \Abs{Q_R}.
\eeq
Suppose not. Then there is $R>0$ and $\rho<1$ such that
for every $x \in \RR^d$ we have
\beql{tmp-1}
\int_{x+Q_R} f*\delta_{\Lambda} < \rho\Abs{Q_R},
\eeq
for any choice of a packing set $\Lambda$.
Pick $n$ such that $\udens \Lambda^n > \rho' > \rho$ and let $Q$ be a cube of side $N\cdot R$ such that
\beql{tmp-3}
\Abs{\Lambda^n \cap Q} \ge \rho' \Abs{Q} = \rho' (N R)^d.
\eeq
Partitioning $Q$ in translates of $Q_R$ we obtain from \eqref{tmp-1} that
\beql{tmp-2}
\int_Q f*\delta_{\Lambda^n} < \rho\Abs{Q} = \rho (N R)^d.
\eeq
Let $\epsilon>0$ and $\Delta>0$ be such that $\int_{Q_\Delta} f > 1-\epsilon$ and define the cube $Q'$ to
have the same center as $Q$ and have side length $N\cdot R - 2\Delta$ so that the $\ell^1$ distance from $Q'$ to $Q^c$
is $\Delta$.
Observe that $\Abs{\Lambda^n \cap (Q\setminus Q')} \le C \Delta (N R)^{d-1}$, where $C>0$ depends only on $f$.

We have
\begin{align*}
\int_Q f*\delta_{\Lambda^n}
 & \ge (1-\epsilon) \Abs{\Lambda^n \cap Q'}\\
 & \ge (1-\epsilon) \big(\rho'(N R)^d - C\Delta (N R)^{d-1}\big).
\end{align*}
If $\epsilon$ is chosen so that $(1-\epsilon)\rho' > \rho$ then we have a contradiction
with \eqref{tmp-2} if $N$ is sufficiently large.

\end{proof}

\begin{lemma}\label{lm:packing-limits}
Assume \eqref{assumptions}.
Suppose $f+\Lambda^n$ are packings and $\Lambda^n \to \Lambda$ weakly.
Then $f+\Lambda$ is also a packing.
\end{lemma}

\begin{proof}
For any $R>0$ we consider the finite sum
\beql{tmp-a-1}
\sum_{\lambda \in \Lambda \cap Q_R} f(x-\lambda) = \sum_{j=1}^N f(x-\lambda_j),
\eeq
where $\Lambda \cap Q_R = \Set{\lambda_1, \lambda_2, \ldots, \lambda_N}$.
By the weak convergence $\Lambda^n \to \Lambda$
we can find for each $n$ points
$$
\lambda^n_1, \lambda^n_2, \ldots, \lambda^n_N \in \Lambda^n
$$
such that $\lambda^n_j \to \lambda_j$ as $n \to \infty$, for $j=1,2,\ldots,N$.
Therefore the expression in \eqref{tmp-a-1} is the limit of
$$
\sum_{j=1}^N f(x-\lambda^n_j),
$$
as $n \to \infty$, which is at most 1 as $\Lambda^n$ are packing sets.
Since $R>0$ is arbitrary we conclude that $f+\Lambda$ is a packing.
\end{proof}


\begin{lemma}\label{lm:tail}
Assume \eqref{assumptions}.
Suppose $f+\Lambda^n$, $n=1,2,\ldots,$ are packings and $K \subseteq \RR^d$ is a compact set.
Then
\beql{tail}
\lim_{n\to\infty} \int_K \sum_{\lambda \in \Lambda^n \setminus Q_n} f(x-\lambda) = 0.
\eeq
\end{lemma}

\begin{proof}
Since the $f+\Lambda^n$ are packings it follows that there exists $\delta_0>0$ such that
the elements of any $\Lambda^n$ have a minimum distance $\ge \delta_0$ (this is a consequence of the last property in \eqref{assumptions}).
It follows that
there exists a positive constant $C$ so that
each point $x \in \RR^d$ is contained in at most $C$ of the sets
$$
\lambda+K,\ \ \ (\lambda \in \Lambda^n)
$$
for any $n$.
Then
\begin{align*}
\int_K \sum_{\lambda \in \Lambda^n \setminus Q_n} f(x-\lambda)
  &= \sum_{\lambda \in \Lambda^n \setminus Q_n} \int_{K+\lambda} f(x)\\
  &\le C \int_{Q_{n/2}^c} f(x),
\end{align*}
if $n$ is sufficiently large.
Since $f$ is integrable the latter integral can be made arbitrarily small if $n$
is sufficiently large.
\end{proof}

\begin{lemma}\label{lm:packing-to-tiling}
Assume \eqref{assumptions}.
Suppose $f+\Lambda^n$ are packings and $\Lambda^n \to \Lambda$ weakly.
If $K$ is a compact set and $\int_K f*\delta_{\Lambda^n} \to \Abs{K}$ then
\beql{tiling-on-K}
\int_K f*\delta_\Lambda = \Abs{K}
\eeq
and $f*\delta_\Lambda \equiv 1$ almost everywhere on $K$.
\end{lemma}

\begin{proof}
Let $R>0$ and write
$$
\Set{\lambda_1, \lambda_2, \ldots, \lambda_N} = \Lambda \cap Q_R,
$$
for some positive integer $N$.
We can now choose $\lambda^n_j \in \Lambda^n$, for $j=1,2,\ldots,N$, such that $\lambda^n_j \to \lambda_j$ as $n \to \infty$.

Then
\begin{align}
\int_K f*\delta_\Lambda
 & \ge \int_K \sum_{j=1}^N f(x-\lambda_j) \nonumber\\
 & = \lim_{n\to\infty} \int_K \sum_{j=1}^N f(x-\lambda_j^n)\ \ \ \text{(by the $L^1$ continuity of translation)} \nonumber\\
 & = \lim_{n \to \infty} \left( \int_K f*\delta_{\Lambda^n} - \int_K \sum_{\lambda \in \Lambda^n\setminus\Set{\lambda^n_1,\ldots,\lambda^n_N}} f(x-\lambda^n_j) \right) \nonumber\\
 & = \Abs{K} - \lim_{n\to\infty} \int_K \sum_{\lambda \in \Lambda^n\setminus\Set{\lambda^n_1,\ldots,\lambda^n_N}} f(x-\lambda^n_j) \label{tmp-b-1}
\end{align}
If $n$ is sufficiently large then all points of $\Lambda^n \cap Q_{R/2}$ are among the points $\lambda^n_1,\ldots,\lambda^n_N$ so 
the integral in \eqref{tmp-b-1} is at most
$$
\int_K \sum_{\lambda \in \Lambda^n \cap Q_{R/2}^c} f(x-\lambda)
$$
which tends to 0 as $R\to\infty$, by Lemma \ref{lm:tail}.

Since the limit in \eqref{tmp-b-1} can be made arbitrarily small
it follows that $\int_K f*\delta_\Lambda = \Abs{K}$.
Since $f+\Lambda$ is a packing, from Lemma \ref{lm:packing-limits}, this implies that $f+\Lambda$ is globally a packing and a tiling on $K$.

\end{proof}

\begin{lemma}\label{lm:local-to-global}
Assume \eqref{assumptions}.
If for every $n$ there is a packing $f+\Lambda^n$ for which
\beql{covering}
\int_{Q_n} f*\delta_{\Lambda^n} \ge \Abs{Q_n}-\frac{1}{n},
\eeq
then there is a subsequence of $\Lambda^n$ which converges weakly to a packing set
$\Lambda$.
\end{lemma}

\begin{proof}
Number the elements of $\Lambda^n$ as $\lambda^n_1, \lambda^n_2, \ldots$, in increasing
order of magnitude, breaking ties arbitrarily.
We claim that for all $j=1,2,\ldots$, the sequence
\beql{seq-bounded}
\Abs{\lambda^n_j},\ \ \ n=1,2,3,\ldots,
\eeq
is bounded.
If not then we apply Lemma \ref{lm:tail} with $R = \Abs{\lambda^n_j}$
(which can be arbitrarily large) to get that
$$
\int_{Q_j} \sum_{\lambda \in \Lambda^n,\ \Abs{\lambda}>R} f(x-\lambda)
$$
can be arbitrarily small.
Because of \eqref{covering}
we deduce that the integral
$$
\int_{Q_j} \sum_{k=1}^j f(x-\lambda^n_k)
$$
must be arbitrarily close to $\Abs{Q_j}/2 = j^d/2$, but this is impossible as the above
sum has $j$ terms each of which can contribute at most 1 to the integral.

Since for each $j=1,2,\ldots$ the sequence $\Abs{\lambda^n_j}$ is bounded,
it follows by a standard diagonal argument that there exists a subsequence
of $\Lambda^n$, call it again $\Lambda^n$,  such that for all $j$ the sequence
$\lambda^n_j$ has a limit
$$
\lambda_j = \lim_n \lambda^n_j.
$$
Let $\Lambda = \Set{\lambda_1, \lambda_2, \ldots}$ and observe that
$\Lambda$ is the weak limit of $\Lambda^n$ and $\Lambda$ is a packing set
because of Lemma \ref{lm:packing-limits}.

\end{proof}


\begin{proof}[Proof of Theorem \ref{th:main}]
From Lemma \ref{lm:density-integral} we conclude that $f$ has packings $\Lambda^n$
such that
$$
\int_{Q_n} f*\delta_{\Lambda^n} \ge \Abs{Q_n} - \frac{1}{n}.
$$
Notice that this implies that for any compact $K$ and sufficiently large $n$ we have
$$
\int_K f*\delta_{\Lambda^n} \ge \Abs{K} - \frac{1}{n}.
$$
Lemma \ref{lm:local-to-global} now implies that $\Lambda^n$ has a subsequence, call it $\Lambda^n$ again,
which converges weakly to a packing set $\Lambda$.
Lemma \ref{lm:packing-to-tiling} now shows that $f+\Lambda$ is a tiling on any compact $K$, hence on all of $\RR^d$.

\end{proof}

\section{Product domains which are spectral}
\label{sec:cylinders}

In this section we make use of Theorem \ref{th:main} in order to show that the
spectrality of certain products implies the spectrality of the factors.


\subsection{Orthogonal packing regions}
\label{sec:orthopack}

Suppose that $A \subseteq \RR^m$ is such that for a set $D \subseteq \RR^m$ we have
$$
(D-D) \cap \Set{\ft{\chi_A}=0} = \emptyset.
$$
The set $D$ is called an {\em orthogonal packing region} for $A$. If it is is also true
that $\Abs{D} = \Abs{A}^{-1}$ then $D$ is called a {\em tight orthogonal packing region} for $A$
\cite{kolountzakis2000packing,lagarias2000orthonormal}.

If $\Lambda$ is an orthogonal set of exponentials and $D$ is an orthogonal packing region
for $A$ then, because of \eqref{zeros-condition} we have
$$
(D-D) \cap (\Lambda-\Lambda) = \Set{0},
$$
which implies that $D+\Lambda$ is a packing and, therefore, that
\beql{packing-region-size}
\Abs{D} \le (\udens{\Lambda})^{-1}.
\eeq
If $\Lambda$ is also complete then $\dens\Lambda = \Abs{A}$
so that, in this case, we have
\beql{packing-region-size-spectral}
\Abs{D} \le \Abs{A}^{-1}.
\eeq
Another way to view \eqref{packing-region-size-spectral} is to say that if a set $A$ has an orthogonal
packing region of size $> \Abs{A}^{-1}$ then $A$ cannot be spectral.

\begin{theorem}\label{th:cylinders}
Suppose $\Omega = A \times B \subseteq \RR^m\times \RR^n$ has $\Abs{A} = \Abs{B} = 1$, and suppose also
that the bounded set $D \subseteq \RR^m$ is such that
\beql{region}
(D-D) \cap \Set{\ft{\chi_A}=0} = \emptyset.
\eeq
\begin{itemize}
\item[(a)] If $\Abs{D}=1$ and $\Omega$ is spectral then $B$ is also spectral.
\item[(b)] If $\Abs{D}>1$ then neither $A$ nor $\Omega$ can be spectral.
\end{itemize}
\end{theorem}


\subsection{Proof of Theorem \ref{th:cylinders}}
\label{sec:cyl-proof}


\begin{lemma}\label{lm:window-density}
Suppose $D \subseteq \RR^m$ is a bounded set and that $\Lambda \subseteq \RR^m\times\RR^n$ is
a uniformly discrete set of upper density $\udens\Lambda=\tau>0$.
For $x \in \RR^m$ write
\beql{proj-density-1}
\alpha(x) = \udens \pi_2 \big(\Lambda \cap \left( (x+D)\times \RR^n\big) \right) .
\eeq
Then
\beql{proj-density-2}
\sup_{x \in \RR^m} \alpha(x) \ge \Abs{D}\tau.
\eeq
\end{lemma}

\noindent{\bf Remark:} The subset of $\RR^n$ whose upper density appears in \eqref{proj-density-1} is, in general, a multiset.

\begin{proof}
Suppose \eqref{proj-density-2} is not true.
Then for all $x \in \RR^m$ we have
\beql{ub-1}
\alpha(x) \le \rho\Abs{D},
\eeq
for some positive number $\rho<\tau$.
Let $\epsilon = (\tau-\rho)/2$. By possibly translating $\Lambda$ we may assume that
$$
\Abs{\Lambda \cap Q_R} \ge (\tau-\epsilon) R^{m+n},
$$
where $R$ may be taken arbitrarily large.

Then
\begin{align}
\int_{\pi_1 Q_R} \Abs{\Lambda \cap \big( (x+D)\times\pi_2 Q_R \big)}\,dx
 &= \sum_{\lambda \in \Lambda\cap Q_R} \Abs{\Set{x\in\RR^m:\ \pi_1\lambda \in x+D}} - O(R^{m+n-1}) \label{truncation}\\
 &= \Abs{D} \Abs{\Lambda \cap Q_R} - O(R^{m+n-1}) \nonumber\\
 &\ge \Abs{D}(\tau-\epsilon)R^{m+n} - O(R^{m+n-1}) \label{lb-1}.
\end{align}

(The error term $O(R^{m+n-1})$ in \eqref{truncation} is due to the boundary of the cube $Q_R$ combined with
the assumed uniform discreteness of $\Lambda$.)

Thus there exists $x \in \pi_1 Q_R$ such that
$$
\Abs{\Lambda \cap \big( (x+D)\times \pi_2 Q_R \big) } \ge \Abs{D} (\tau-\epsilon)R^n + O(R^{n-1}).
$$
This contradicts \eqref{ub-1} if $R$ is sufficiently large.
\end{proof}

\begin{proof}[Proof of Theorem \ref{th:cylinders}]
(a) By \eqref{tiling-condition} and Theorem \ref{th:main}
it suffices to exhibit packings of the function $\Abs{\ft{\chi_B}}^2$ of upper density arbitrarily close to 1.
Suppose $\epsilon>0$ and suppose also that $\Lambda$ is a spectrum for $\Omega$, and therefore
$\Lambda$ has density 1.
From Lemma \ref{lm:window-density} there exists $a \in \RR^m$ such that the multiset
$$
L = \pi_2 \left( \Lambda \cap \big( (a+D)\times \RR^n\big) \right) \subseteq \RR^n
$$
has upper density at least $1-\epsilon$.

We claim that $L$ is an orthogonal set (not multiset) for $B$, hence that $\Abs{\ft{\chi_B}}^2+L$ is a packing
of upper density $\ge 1-\epsilon$ by \eqref{packing-condition}. Suppose $x, y \in L$ are two distinct points in $L$.
This means that there are points $d_1, d_2 \in D$ such that
$$
(a+d_1, x), (a+d_2, y) \in \Lambda.
$$
Since $\ft{\chi_\Omega}(\xi,\eta) = \ft{\chi_A(\xi)} \ft{\chi_B(\eta)}$ and $(a+d_1 ,x), (a+d_2, y)$ are orthogonal for $\Omega$,
it follows that we must have
$$
d_1-d_2 \in \Set{\ft{\chi_A}=0}\ \ \text{or}\ \ x-y \in \Set{\ft{\chi_B}=0}.
$$
But the first alternative cannot hold by our assumption \eqref{region} on $D$, hence we conclude that $x, y$ are orthogonal for $B$.
By the same reasoning we conclude that $L$ is a set. Indeed, if there are two distinct points in 
$$
\Lambda \cap \big( (a+D)\times \RR^n \big)
$$
which project down to the same element of $L$, call them $(a+d_1, x)$ and $(a+d_2, x)$ we get that their
difference $(d_1-d_2, 0)$ is not a point of vanishing of $\ft{\chi_\Omega}$, a contradiction.

Since $\epsilon$ is arbitrarily small we have exhibited packings of $\Abs{\ft{\chi_B}}^2$ of density arbitrarily close to 1.

(b) That $A$, of volume 1,  cannot be spectral if it has an orthogonal packing region of volume $>1$ has been explained at the beginning of
\S\ref{sec:orthopack}. Assume, as in (a), that $\Omega$ is spectral with spectrum $\Lambda$. Then $\dens\Lambda=1$.
The set $L$ constructed in (a) now has density > 1, and, by the reasoning of (a), $L$ is an orthogonal set of exponentials for $B$,
a contradiction since $\Abs{B} = 1$. Hence $\Omega$ cannot be spectral, as we had to prove.
\end{proof}

We can now obtain the result of \cite{greenfeld2016spectrality}.
\begin{corollary}\label{cor:interval}
Suppose that the set $[0,1]\times B \subseteq \RR^{1+n}$ is spectral.
Then so is $B\subseteq\RR^n$.
\end{corollary}

\begin{proof}
Apply Theorem \ref{th:cylinders}(a) with $A=[0,1]$, $D=(-1/2,1/2)$.
\end{proof}

\begin{corollary}\label{cor:interval-d}
Suppose that the set $[0,1]^d\times B \subseteq \RR^{d+n}$ is spectral.
Then so is $B\subseteq\RR^n$.
\end{corollary}
\begin{proof}
Apply Theorem \ref{th:cylinders}(a) with $A=[0,1]^d$, $D=(-1/2,1/2)^d$.
However Corollary \ref{cor:interval-d} can also easily be derived from Corollary \ref{cor:interval} by induction on $d$.
\end{proof}

Our next result extends the result of \cite{greenfeld2016spectrality} to the union of two intervals.
\begin{corollary}\label{cor:two-intervals}
Suppose that the set $(I \cup J) \times B \subseteq \RR^{1+n}$ is spectral, where $I, J$ are two disjoint closed intervals.
Then both $I\cup J \subseteq \RR$ and $B\subseteq\RR^n$ are spectral.
\end{corollary}

We shall need the following lemma proved in \cite{aak} (though not exactly in this form).

\begin{lemma}\label{lm:first-zero}
Let $I$ and $J$ be two disjoint closed intervals, satisfying $|I|+|J| = 1$ and define
$$
A = I \cup J.
$$
\begin{itemize}
\item[(a)]
If $\Abs{I} \neq \Abs{J}$ then
we have that $\widehat{\chi}_A(x)\neq 0$ for every $x\in(-1,1)$.
\item[(b)]
If $\Abs{I} = \Abs{J} = \nicefrac{1}{2}$ then the zero set of $\widehat{\chi}_A$ is
\beql{zset}
Z = \Set{\widehat{\chi}_A=0} = 2\ZZ\setminus{0} \cup (2\ZZ+1)\Delta,
\eeq
where $\Delta = \frac{1}{2\Abs{m_1-m_2}}<1$ and $m_1, m_2$ are the midpoints of $I$ and $J$.
\end{itemize}
\end{lemma}

\begin{proof}
We can consider $x\neq 0$, since clearly $\widehat{\chi_A}(0)= 1$.
Note that given $a,b\in \RR$,  we have
$$
\widehat{\chi_{[a,b]}}(x) = \frac{\sin \pi (b-a) x}{\pi x}  \, \, e^{-\pi i (a+b) \,x}.
$$
Let $\ell_1=|I|$, $\ell_2=|J|$, and  let $m_1$,  $m_2$ be the midpoints of $I$ and $J$ respectively. Then
\beql{ft2}
\widehat{\chi_A}(x)
  = e^{-2\pi i m_1 x}   \frac{\sin \pi \ell_1 x}{\pi x}
      + e^{-2 \pi i  m_2  x} \frac{\sin \pi \ell_2 x}{\pi x}.
\eeq

Setting \eqref{ft2} equal to 0 we get the necessary condition for vanishing at $x$
$$
|\sin (\pi \ell_1 x) | =|\sin (\pi \ell_2 x)|.
$$
Suppose $\ell_1 \neq \ell_2$.
Then for $0<x<1$ (and similarly for $-1<x<0$)
this is impossible, since the two angles $\pi\ell_1 x$ and $\pi\ell_2 x$ have sum $\pi x < \pi$
and they are not identical.
Hence, $\widehat{\chi_A}(x)\neq 0$ for $x \in (-1,1)$, proving part (a) of the Lemma.

In the case when $\ell_1 = \ell_2 = 1/2$ \eqref{ft2} becomes
\beql{ft2new}
\widehat{\chi_A}(x) = \frac{\sin \pi x/2}{\pi x}(e^{-2\pi i m_1 x}+e^{-2\pi i m_2 x})
\eeq
which vanishes precisely at the set $Z$ in \eqref{zset},
the first part in the union \eqref{zset} due to the sine factor in \eqref{ft2new}
and the second part due to the sum of two exponentials in \eqref{ft2new}.
\end{proof}

\begin{proof}[Proof of Corollary \ref{cor:two-intervals}]

We assume, as we may, that $\Abs{B} = \Abs{I}+\Abs{J} = 1$.

\noindent \underline{Case 1:} $\Abs{I} \neq \Abs{J}$.

By Lemma \ref{lm:first-zero}(a) we have that, if $D=(-1/2,1/2)$, then $D-D = (-1, 1)$ does not intersect
the zeros of $\ft{\chi_A}$, where $A=I\cup J$.
An application then of Theorem \ref{th:cylinders}(a) gives that $B$ is spectral.

To see that $I \cup J$ is also spectral we first observe that can discount the case when $\ft{\chi_A}(1) \neq 0$.
Indeed, in that case the interval $D$ can be taken to be properly longer than 1, namely $D=(-\frac12-\epsilon, \frac12+\epsilon)$ for some $\epsilon>0$,
and this would give, using Theorem \ref{th:cylinders}(b),
that $(I\cup J)\times B$ is not spectral, a contradiction.

Assuming, therefore, that $\ft{\chi_A}(1)=0$ we obtain easily from \eqref{ft2} that
$\Abs{m_1-m_2} = k+\frac{1}{2}$, for some positive integer $k$, where, again, $m_1, m_2$ are the midpoints of the two
intervals. But this implies that the set $A$ tiles with $\ZZ$, therefore $A$ is spectral with spectrum $\ZZ$.

\ \\

\noindent \underline{Case 2:} $\Abs{I} = \Abs{J} = 1/2$.

We now use Lemma \ref{lm:first-zero}(b) and define
$$
D = (0,2) \cap \bigcup_{n=0}^\infty (2n\Delta, (2n+1)\Delta),
$$
where $\Delta = \frac{1}{2\Abs{m_1-m_2}}$ as defined in Lemma \ref{lm:first-zero}.
Observe first that $D-D$ does not contain any of the zeros of $\ft{\chi_A}$, which are given in \eqref{zset}.
Notice also that $\Abs{D} \ge 1$ with equality precisely when $\Delta$ divides 1, or, equivalently, $\Abs{m_1-m_2} \in \frac12\ZZ$.

Again, because of Theorem \ref{th:cylinders}(b) the case $\Abs{D}>1$ cannot occur.

So we must have $\Abs{D}=1$, in which case Theorem \ref{th:cylinders}(a) proves that $B$ is spectral.
This happens only when $\Abs{m_1-m_2}$ is an integer or half-integer. In this case the set $A$ is also
spectral and tiles the line too (see e.g.\ \cite{laba2001twointervals} where it is shown that for sets which are
unions of two intervals the Fuglede conjecture holds true; it is very easy to see that our set $A$ tiles).

\end{proof}


\printbibliography

\end{document}